\documentclass[11pt]{article}
\usepackage[latin1]{inputenc}
\usepackage{amsmath,amsthm,amssymb}
\usepackage{amsfonts}
\usepackage{amsmath,amsthm,amssymb,amscd}
\usepackage{latexsym}
\usepackage{color}
\usepackage{graphicx}
\usepackage{mathrsfs}

\makeatletter \@addtoreset{equation}{section} \makeatother

\newtheorem{theorem}{Theorem}[section]

\newtheorem{proposition}{Proposition}[section]
\newtheorem{lemma}{Lemma}[section]

\begin{document}
\title{Fractional magnetic Schr\"{o}dinger-Kirchhoff problems
\\with convolution and  critical nonlinearities}
\author{Sihua Liang$^a$, Du\v{s}an D. Repov\v{s}$^{b,c}$, and Binlin Zhang$^{d}$\footnote{Corresponding author.
E-mail address:\, liangsihua@126.com (S. Liang), dusan.repovs@guest.arnes.si (D.~D. Repov\v{s}), zhangbinlin2012@163.com (B. Zhang)}\\
\footnotesize $^a$ College of Mathematics, Changchun Normal
University, Changchun, 130032, P.R. China\\
\footnotesize $^b$ 
University of Ljubljana, Ljubljana, 1000, Slovenia\\ \footnotesize
\footnotesize $^c$ Institute of Mathematics, Physics and Mechanics, Ljubljana, 1000, Slovenia\\
\footnotesize $^d$ College of Mathematics and System Science, Shandong University of Science and Technology,
\\ \footnotesize Qingdao, 266590, P.R. China}
\date{ }
\maketitle

\begin{abstract}
{In this paper we are concerned with the existence and multiplicity of solutions for
the  fractional Choquard-type Schr\"{o}dinger-Kirchhoff equations with electromagnetic fields and
critical nonlinearity:
\begin{eqnarray*}
\begin{cases}
\varepsilon^{2s}M([u]_{s,A}^2)(-\Delta)_{A}^su + V(x)u =
(|x|^{-\alpha}*F(|u|^2))f(|u|^2)u
+ |u|^{2_s^\ast-2}u,\ \ \
x\in  \mathbb{R}^N,\\
u(x) \rightarrow 0,\ \  \quad \mbox{as}\ |x| \rightarrow \infty,
\end{cases}
\end{eqnarray*}
where $(-\Delta)_{A}^s$ is the fractional magnetic
operator with $0<s<1$, $2_s^\ast = 2N/(N-2s)$, $\alpha < \min\{N,
4s\}$, $M : \mathbb{R}^{+}_{0}\rightarrow \mathbb{R}^{+}_0$ is a
continuous function, $A: \mathbb{R}^N \rightarrow \mathbb{R}^N$ is
the magnetic potential, $F(|u|) = \int_0^{|u|}f(t)dt$,
and
 $\varepsilon
> 0$ is a positive parameter. The  electric potential  $V\in C(\mathbb{R}^N, \mathbb{R}^+_0)$ satisfies $V(x) = 0$ in some region of
$\mathbb{R}^N$, which means that
this is
 the critical frequency case. We
first prove the $(PS)_c$ condition, by using the fractional version
of the concentration compactness principle. Then,
applying also
the
mountain pass theorem and the genus theory, we
obtain the existence
and multiplicity of semiclassical states  for the above problem. The
main feature of our problems is that
the Kirchhoff term $M$ can vanish at zero. }\medskip

\medskip

\end{abstract}

\section{Introduction and main results}
In this paper we consider  the
fractional Choquard-Kirchhoff type problem with electromagnetic fields and
critical nonlinearity:
\begin{equation}\label{e1.1}
\left\{
\begin{array}{lll}
\varepsilon^{2s}M([u]_{s,A}^2)(-\Delta)_{A}^su + V(x)u =
(\mathcal{K}_\alpha*F(|u|^2))f(|u|^2)u + |u|^{2_s^\ast-2}u,\
x\in  \mathbb{R}^N,\smallskip\smallskip\\
u(x) \rightarrow 0, \quad \mbox{as}\ |x| \rightarrow \infty,
\end{array}\right.
\end{equation}
where $\varepsilon > 0$ is a positive parameter, $N >2s$, $0 < s <
1$, $2_s^\ast = 2N/(N-2s)$ is the critical Sobolev exponent,
$V\in C(\mathbb{R}^N, \mathbb{R}^+_0)$ is an electric potential,
$\mathcal{K}_\alpha(x) = |x|^{-\alpha}$, $\alpha < \min\{N, 4s\}$,
$A\in C(\mathbb{R}^N, \mathbb{R}^N)$ is a magnetic potential, and
\begin{eqnarray*}
[u]_{s,A}^2 := \iint_{\mathbb{R}^{2N}}\frac{|u(x)-e^{i(x-y)\cdot
A(\frac{x+y}{2})}u(y)|^2}{|x-y|^{N+2s}}dx dy.
\end{eqnarray*}
If $A$ is a smooth function, the fractional operator
$(-\Delta)_{A}^s$, which up to normalization constants can be
defined on smooth functions $u$ as
\begin{eqnarray*}
(-\Delta)_{A}^s u(x) := 2\lim_{\varepsilon \rightarrow 0}
\int_{\mathbb{R}^N \setminus
B_\varepsilon(x)}\frac{u(x)-e^{i(x-y)\cdot
A(\frac{x+y}{2})}u(y)}{|x-y|^{N+2s}}dy, \quad x\in  \mathbb{R}^N,
\end{eqnarray*}
has recently been introduced in \cite{fel}. Hereafter,
$B_\varepsilon(x)$ denotes the ball of $\mathbb{R}^N$ centered at $x
\in \mathbb{R}^N$ and of radius $\varepsilon > 0$. For details on
fractional magnetic operators we refer to \cite{fel}, and for the physical background we refer to
 \cite{I1, I3, I4}.

This paper was motivated by some works
concerning the magnetic Schr\"{o}dinger equation
\begin{equation}\label{101}
-(\nabla u- i A)^2u + V(x)u = f(x, |u|)u,
\end{equation}
which have appeared in recent years  (see \cite{ar, ci, dic, ku, sq})
and have extensively studied \ref{101},
when the above magnetic operator is defined as
\begin{displaymath}
-(\nabla u- i A)^2u = -\Delta u +2i A(x)\cdot\nabla u +
|A(x)|^2u + i u \mbox{div} A(x).
\end{displaymath}
As stated in \cite{sq1}, up to correcting the operator by the factor
$(1-s)$, it follows that $(-\Delta)^s_A u$ converges to $-(\nabla
u-i A)^2u$ as $s\rightarrow1$.

Thus, up to normalization, the
nonlocal case can be seen as an approximation of the local one. The
motivation for its introduction was
 described in \cite{fel,sq1}  and relies essentially on the L\'{e}vy-Khintchine formula
for the generator of a general L\'{e}vy process. If the magnetic
field $A \equiv 0$, the operator $(-\Delta)_{A_\varepsilon}^s$ can
be reduced to the fractional Laplacian operator $(-\Delta)^s$, which
 may be viewed as the
infinitesimal generator of a L$\acute{\mbox{e}}$vy stable diffusion
processes \cite{ap}. This operator arises in the description of
various phenomena in applied sciences, such as phase transitions,
materials science, conservation laws, minimal surfaces, water waves,
optimization, plasma physics, etc., see \cite{di} and references
therein.

The study of
fractional and nonlocal operators of elliptic type  has recently attracted
a lot of
 attention. For the cases in which bounded domains and
the entire space are involved, we refer the readers e.g. to
\cite{ADM, BMS, MRS, m2, YW, ZDM} and the references therein.
When the interaction between the particles is
considered, i.e., when the nonlinear term $f(u)$ is of type $(\mathcal{K}_\alpha*|u|^p)|u|^{p-2}u$,
 this type of problem is usually called
 the Choquard-type equation and has been
investigated by many authors, see e.g. \cite{Lu, WX}.

Another strong motivation  for studying problem \eqref{e1.1}
is the significant feature of Kirchhoff-type problems. More precisely,  in 1883 Kirchhoff proposed the following model
\begin{align}\label{eq3}
\rho\frac{\partial ^2u}{\partial t^2}-\left(\frac{p_0}{\lambda}+\frac{Y}{2L}\int_0^L\left|\frac{\partial u}{\partial x}\right|^2dx\right)\frac{\partial ^2u}{\partial x^2}=0
\end{align}
as a generalization of the well-known D'Alembert's wave equation for free vibrations of elastic strings. Here, $L$ is the length of the string, $\lambda$ is the area of the cross section, $Y$ is the Young modulus of the material, $\rho$ is the mass density, and $p_0$ is the initial tension.
Essentially, Kirchhoff's model takes into account the changes in the length of the string produced by transverse vibrations.
For recent results in this direction, we refer the reader e.g. to \cite{liang2, liang1}.

Recently, Fiscella and  Valdinoci  \cite{fi} first deduced a stationary fractional Kirchhoff
model which considered the nonlocal
aspect of the tension arising from nonlocal measurements of the
fractional length of the string
(see   \cite[Appendix]{fi} for more details).
More precisely,   the following Kirchhoff-type problem involving critical exponent was studied in \cite{fi}:
 \begin{eqnarray}\label{eq13}
\begin{cases}
M([u]^2_s)(-\Delta)^su=\lambda f(x,u)+|u|^{2_s^*-2}u\quad &\mbox{in }\,\,\Omega\\
u=0\quad&\mbox{in}\,\, \mathbb{R}^N\setminus\Omega.
\end{cases}
\end{eqnarray}
where $\Omega$ is an open
bounded domain in $\mathbb{R}^N$. By using the mountain pass theorem and the concentration compactness principle, together with a truncation technique, the existence of non-negative solutions for problem \eqref{eq13} was obtained.

Here we point out that $M(0)>0$ in \eqref{eq13}, this is called the non-degenerate case.
Otherwise, the problem is called degenerate if $M(0)=0$.
In recent years, there has been a lot of interest in studying fractional Kirchhoff-type problems, here we just list some references,
e.g., see \cite{liang3, m2, PXZ} for recent results on the non-degenerate case,
\cite{au, PXZ1, WX, XZR1, XZZ} for recent results on the degenerate case, and  \cite{MPSZ, PXZ3} for  discussions of both cases.

Next, let us mention some enlightening works related to problem \eqref{e1.1}.
Mingqi {\it et al.} \cite{MPSZ} first studied the following Schr$\ddot{\mbox{o}}$dinger-Kirchhoff type equation
involving the fractional $p$-Laplacian and the magnetic operator
\begin{equation}\label{eq11}
M([u]_{s,A}^2)(-\Delta)_A^su+V(x)u=f(x,|u|)u\quad \text{in
$\mathbb{R}^N$},
\end{equation}
where the right-hand term in \eqref{eq11} satisfies the subcritical
growth. By using  variational methods, they obtained several
existence results for problem \eqref{eq11}. Using  similar
methods, for $M(t)=a+bt$ with $a\in \mathbb{R}^+_0$ and $b\in \mathbb{R}^+$,
Wang and Xiang \cite{WX} proved the existence of two
solutions and  infinitely many solutions for fractional
Schr$\ddot{\mbox{o}}$dinger-Choquard-Kirchhoff type equations with
external magnetic operator and critical exponent in the sense of the
Hardy-Littlewood-Sobolev inequality.

 Binlin
{\it et al.}
 \cite{zhang3}
 first considered the following singularly perturbed
fractional Schr$\ddot{\mbox{o}}$dinger equations:
\begin{eqnarray}\label{eq12}
\varepsilon^{2s}(-\Delta)^{s}_{A_{\varepsilon}}u+V(x)u=f(x,|u|)u+K(x)|u|^{2_{\alpha}^{*}-2}u\quad\quad
\mbox{in}\ \mathbb{R}^{N},
\end{eqnarray}
where $V(x)$ satisfies some assumptions. By using variational
methods, they proved the existence of  solutions $u_{\varepsilon}$ which tends to the
trivial solution as $\varepsilon\rightarrow0$. Moreover, they proved
the existence of infinite many solutions and sign-changing solutions
for problem \eqref{eq12} without magnetic field under some additional assumptions.

Subsequently, Liang {\it et al.}
 \cite{liang3} investigated the  existence and multiplicity of solutions for
problem \eqref{e1.1} without Choquard-type term in the non-degenerate Kirchhoff case.
Very recently, by employing variational methods,
 Ambrosio  \cite{A} obtained the existence and concentration of nontrivial solutions for
a singularly perturbed fractional Choquard problem with a subcritical nonlinearity and an external magnetic field.

 Inspired by the above works, in particular
\cite{zhang3, d3, liang3,  MPSZ}, we consider in this article the
existence and multiplicity of solutions for the fractional Choquard-type
problems with electromagnetic fields and  critical nonlinearity in
the possibly degenerate Kirchhoff context. It is worthwhile to remark that in the arguments
developed in \cite{zhang3, d3}, one of the key points is to prove
the  $(PS)_c$ condition. Here we use the fractional version of
Lions' second concentration compactness principle and concentration
compactness principle at infinity to prove that the $(PS)_c$
condition holds, which is different from methods used in
\cite{zhang3, d3}.

In fact, the appearance of the magnetic field
also brings additional difficulties into the study of our problem, e.g., the
effects of the magnetic fields on the linear spectral sets and on
the solution structure, and the possible interactions between the
magnetic fields and the linear potentials. Therefore, we need to
develop new techniques to conquer difficulties induced by these new
features as well as the possibly degenerate nature of the Kirchhoff
coefficient.

\indent Suppose that  functions $V(x)$, $M(t)$ and
$f(t)$ satisfy the following conditions:
\begin{itemize}
\item[($V$)]  $V(x) \in C(\mathbb{R}^N, \mathbb{R})$, $\min_{x\in \mathbb{R}^N} V(x) = 0$ and there is $\tau_0 > 0$ such that the set $V^{\tau_0} = \{x \in \mathbb{R}^N: V(x) < \tau_0\}$
has finite Lebesgue measure.
\item[($M$)] ($M_1$) there exists $\sigma \in (1, 2_s^\ast/2)$ satisfying
$\sigma\widetilde{M}(t)\geq M(t)t$ for all $t\geq0$, where
$\widetilde{M}(t)=\int_0^tM(s)ds$;\\
($M_2$) there exists $m_1 > 0$ such that $M(t) \geq m_1
t^{\sigma-1}$ for all $t \in \mathbb{R}^+$ and $M(0) = 0$.
\item[($F$)]  ($f_1$) $f \in C(\mathbb{R}^+, \mathbb{R})$; \\
($f_2$) there exist $c_0 > 0$ and $\max\{\sigma,2\} < p <2_s^\ast$
such that $|f(t)|
\leq c_0|t|^{\frac{p-1}{2}}$;\\
($f_3$)  there exist $2\sigma<  \mu < 2_s^\ast$ such that $0 < \mu
F(t) \leq f(t)t$ for all $t \in \mathbb{R}^+$, where $F(t) =
\int_0^t f(s)ds$.
\end{itemize}

The following is our first main result, the existence
 theorem for problem \eqref{e1.1}.
\begin{theorem}\label{the3.1} Let the conditions {\rm ($V$), ($M$)} and {\rm ($F$)}  be satisfied. Then  for any $\kappa > 0$, there is $\mathcal {E}_\kappa > 0$ such
that if $0 < \varepsilon <  \mathcal {E}_\kappa$, then problem
\eqref{e1.1} has at least one solution $u_\varepsilon$  satisfying
\begin{eqnarray}\label{e1.8}
\frac{2\mu-\sigma}{4\sigma}\iint_{\mathbb{R}^{2N}}\frac{F(|u_\varepsilon(x)|^2)F(|u_\varepsilon(y)|^2)}{|x-y|^{\alpha}}dxdy
+\left(\frac{1}{2\sigma}-\frac{1}{2_s^\ast}\right)\int_{\mathbb{R}^N}
|u_\varepsilon|^{2_s^\ast}dx \leq
\kappa{\varepsilon^{\frac{2s2_s^\ast}{2_s^\ast-4\sigma}}},
\end{eqnarray}
\begin{eqnarray}\label{e1.9}
\left(\frac{1}{2\sigma}-\frac{1}{\mu}\right)\alpha_0\varepsilon^{2s}
[u_{\varepsilon}]_{s,A}^{2\sigma} +
\left(\frac{1}{2}-\frac{1}{\mu}\right)
\int_{\mathbb{R}^N}V(x)|u_{\varepsilon}|^2dx \leq
\kappa{\varepsilon^{\frac{2s2_s^\ast}{2_s^\ast-4\sigma}}}.
\end{eqnarray}
Moreover, $u_\varepsilon \rightarrow 0$ in $E$ as $\varepsilon
\rightarrow 0$.
\end{theorem}
The following is our second main result, the multiplicity
 theorem for problem \eqref{e1.1}.
\begin{theorem}\label{the3.2} Let the conditions {\rm ($V$), ($M$)} and {\rm ($F$)}  be satisfied. Then for any $m \in \mathbb{N}$ and $\kappa > 0$, there
is $\mathcal {E}_{m\kappa}
> 0$ such that if $0 < \varepsilon < \mathcal {E}_{m\kappa}$, then problem \eqref{e1.1} has at least $m$ pairs of
solutions $u_{\varepsilon,i}$, $u_{\varepsilon,-i}$,
$i=1,2,\cdots,m$ which satisfy estimates \eqref{e1.8} and
\eqref{e1.9}. Moreover, $u_{\varepsilon,i}\rightarrow 0$ in $E$ as
$\varepsilon \rightarrow 0$, $i=1,2,\cdots,m$.
\end{theorem}

\section{Functional setting}\label{sec2}

\indent In this paper, we shall use Banach space $E$ defined by
\begin{displaymath}
E =\left\{u\in H_{A}^{s}(\mathbb{R}^N,\mathbb{C}):
\int_{\mathbb{R}^N}V(x)|u|^{2}dx<\infty \right\}
\end{displaymath}
with the norm
\begin{displaymath}
\|u\|_E := \left([u]_{s,A}^2 +
\int_{\mathbb{R}^N}V(x)|u|^2dx\right)^{\frac{1}{2}},
\end{displaymath}
where $V$ is non-negative,
$H_{A_\varepsilon}^{s}(\mathbb{R}^N,\mathbb{C})$ is the fractional
Sobolev space defined by
\begin{displaymath}
H_{A}^{s}(\mathbb{R}^N,\mathbb{C})  =  \left\{u \in
L^2(\mathbb{R}^N,\mathbb{C}): [u]_{s,A} < \infty\right\},
\end{displaymath}
where  $s \in (0, 1)$ and $[u]_{s,A}$ denotes the so-called
Gagliardo semi-norm, that is
\begin{displaymath}
[u]_{s,A}  = \left(\iint_{\mathbb{R}^{2N}}\frac{|u(x)-e^{i(x-y)\cdot
A(\frac{x+y}{2})}u(y)|^2}{|x-y|^{N+2s}}dx dy\right)^{1/2}
\end{displaymath}
and $H_{A}^{s}(\mathbb{R}^N,\mathbb{C})$ is endowed with the norm
\begin{displaymath}
\|u\|_{H_{A}^{s}(\mathbb{R}^N,\mathbb{C})} = \left([u]_{s,A}^2 +
\|u\|_{L^2}^2\right)^{\frac{1}{2}}.
\end{displaymath}

By assumption $(V)$, we know that the embedding $E
\hookrightarrow H_{A}^{s}(\mathbb{R}^N,\mathbb{C})$ is continuous.
Note that the norm $\|\cdot\|_E$ is equivalent to the norm
$\|\cdot\|_\varepsilon$ defined by
\begin{displaymath}
\|u\|_\varepsilon := \left([u]_{s,A}^2 +
\varepsilon^{-2s}\int_{\mathbb{R}^N}V(x)|u|^2dx\right)^{\frac{1}{2}},
\end{displaymath}
for each $\varepsilon > 0$.  It is obvious that for each $ \theta
\in [2, 2_s^\ast]$, there is $c_{\theta}>0$, independent of $0 <
\varepsilon < 1$, such that
\begin{equation}\label{e2.2}
|u|_{\theta} \leq c_{\theta}\|u\|_E \leq
c_{\theta}\|u\|_{\varepsilon}.
\end{equation}
 Hereafter, we shortly denote by $\|\cdot\|_\nu$ the norm of Lebesgue space $L^\nu(\Omega)$ with $\nu\geq1$.

\indent We first recall the following embedding theorem:
\begin{proposition}\label{pro1}{\em (see \cite[Lemma 3.5]{fel})}. Let $A\in C(\mathbb{R}^N, \mathbb{R}^N)$. Then the embedding
\begin{displaymath}
H_{A}^{s}(\mathbb{R}^N,\mathbb{C}) \hookrightarrow
L^{\theta}(\mathbb{R}^N,\mathbb{C}),
\end{displaymath}
 is continuous for any $\theta \in [2,2_s^\ast]$. Moreover, the
 embedding
\begin{displaymath}
H_{A}^{s}(\mathbb{R}^N,\mathbb{C}) \hookrightarrow
L^{\theta}_{{\rm loc}}(\mathbb{R}^N,\mathbb{C})
\end{displaymath}
is compact for any $\theta \in [1,2_s^\ast)$.
\end{proposition}

\indent We shall use the following diamagnetic inequality:
\begin{lemma}\label{lemma2.1}{\em (see \cite[Lemma 3.3]{fel})}.  For every $u \in
H_{A}^{s}(\mathbb{R}^N,\mathbb{C})$,
$$|u| \in
H^{s}(\mathbb{R}^N).$$ More precisely,
\begin{displaymath}
[|u|]_{s} \leq [u]_{s,A}.
\end{displaymath}
\end{lemma}

\indent By Proposition 3.6 in \cite{di}, we have
\begin{displaymath}
[u]_{s} = \|(-\Delta)^{\frac{s}{2}}\|_{L^2(\mathbb{R}^N)}
\end{displaymath}
for any $u \in H^{s}(\mathbb{R}^N)$, i.e.
\begin{displaymath}
\iint_{\mathbb{R}^{2N}}\frac{|u(x)-u(y)|^2}{|x-y|^{N+2s}}dxdy  =
\int_{\mathbb{R}^{N}}|(-\Delta)^{\frac{s}{2}}u(x)|^2dx.
\end{displaymath}
Thus
\begin{displaymath}
\iint_{\mathbb{R}^{2N}}\frac{(u(x)-u(y))(v(x)-v(y))}{|x-y|^{N+2s}}dx
dy  =
\int_{\mathbb{R}^{N}}(-\Delta)^{\frac{s}{2}}u(x)\cdot(-\Delta)^{\frac{s}{2}}v(x)
dx.
\end{displaymath}

To obtain the solution of problem \eqref{e1.1}, we shall use the
following equivalent form
\begin{equation}\label{e3.1}
\left\{
\begin{array}{lll}M\Big([u]_{s,A}^2\Big)(-\Delta)_{A}^su + \varepsilon^{-2s} V(x)u =
 \varepsilon^{-2s}\displaystyle\int_{\mathbb{R}^N}\frac{F(|u|^2)}{|x-y|^{\alpha}}dyf(|u|^2)u + \varepsilon^{-2s}|u|^{2_s^\ast-2}u, \,
x\in  \mathbb{R}^N,\smallskip\smallskip\\
u(x) \rightarrow 0, \ \indent as \ \ |x| \rightarrow \infty,
\end{array}\right.
\end{equation}
for $\varepsilon \rightarrow 0$. \\
\indent The energy functional $J_\varepsilon: E \rightarrow
\mathbb{R}$ associated with problem  \eqref{e3.1}
\begin{eqnarray*}
J_\varepsilon(u) :=\frac12 \widetilde{M}\left([u]_{s,A}^2\right)+
\frac{\varepsilon^{-2s}}{2}\int_{\mathbb{R}^N}V(x)|u|^2dx-
\frac{\varepsilon^{-2s}}{4}\iint_{\mathbb{R}^{2N}}\frac{F(|u(x)|^2)F(|u(y)|^2)}{|x-y|^{\alpha}}dxdy
- \frac{\varepsilon^{-2s}}{2_s^\ast}\int_{\mathbb{R}^N}
|u|^{2_s^\ast}dx
\end{eqnarray*}
is well defined. Under the assumptions, it is easy to check that as
shown in \cite{r1,w1}, $J_\varepsilon \in C^1 (E, \mathbb{R})$ and
its critical points are weak solutions of problem  \eqref{e3.1}.
\\
 \indent By condition $(f_2)$,  we have
\begin{eqnarray*}
F(|u|^2) \leq C(|u|^2 + |u|^p), \quad  \mbox{for all}  \ \ u \in
H_A^s(\mathbb{R}^N, \mathbb{C}).
\end{eqnarray*}
Note that, by the Hardy-Littlewood-Sobolev inequality, the integral
\begin{eqnarray*}
\iint_{\mathbb{R}^{2N}}\frac{F(|u(x)|^2)F(|u(y)|^2)}{|x-y|^{\alpha}}dxdy
\end{eqnarray*}
is well defined if $F(|u|^2) \in L^r(\mathbb{R}^N)$ for some $r > 1$
satisfying
\begin{eqnarray*}
\frac{2}{r} + \frac{\alpha}{N} = 2,
\end{eqnarray*}
that is $r = 2N/(2N - \alpha)$. Actually, by $\alpha <
\min\{N, 4s\}$, it follows that $2 < 2r < 2_s^\ast$. Moreover, from
$2 < pr < 2_s^\ast$, we
can
 deduce
\begin{eqnarray}
\int_{\mathbb{R}^N}|F(|u|^2)|^rdx  &\leq&
2^{r-1}C^r\left(\int_{\mathbb{R}^N}|u|^{2r}dx +
\int_{\mathbb{R}^N}|u|^{pr}dx\right) \\
&\leq& 2^{r-1}C^r\left(C_{2r}^{2r}\|u\|^{2r} +
C_{pr}^{pr}\|u\|^{pr}\right)  \quad \mbox{for all} \  \ u \in
H_A^s(\mathbb{R}^N, \mathbb{C}).
\end{eqnarray}
By a standard
argument, one can show that $J_\varepsilon(u)$ is of
class $C^1$ and
\begin{eqnarray*}
\langle J_\varepsilon'(u), v\rangle &=&
M\left([u]_{s,A}^2\right)\mbox{Re}\iint_{\mathbb{R}^{2N}}\frac{(u(x)-e^{i(x-y)\cdot
A(\frac{x+y}{2})}u(y))\overline{(v(x)-e^{i(x-y)\cdot
A(\frac{x+y}{2})}v(y))}}{|x-y|^{N+2s}}dx dy  \\
&& \mbox{}\ + \varepsilon^{-2s}\mbox{Re}
\int_{\mathbb{R}^N}V(x)u\bar{v}dx - \varepsilon^{-2s} \mbox{Re}
\int_{\mathbb{R}^N}(\mathcal{K}_\mu*F(|u|^2))f(|u|^2)u\bar{v}dx -
\varepsilon^{-2s} \mbox{Re}
\int_{\mathbb{R}^N}|u|^{2_s^\ast-2}u\bar{v}dx,
\end{eqnarray*}
for all $u, v \in E$. Hence a critical point of $J_\varepsilon$ is a weak solution of problem \eqref{e1.1}.\\
\indent  Now we recall the general version of the mountain pass
theorem in  \cite{r1}
 which will be used later.
\begin{theorem}\label{the5.1} Let $\mathcal{J}$ be a functional on a Banach
space $Y$ and $\mathcal{J} \in C^1(Y, \mathbb{R})$. Let us assume that there
exist $\zeta, \rho > 0$ such that
\begin{itemize}
\item[$(i)$] $\mathcal{J}(u) \geq \zeta$, for every $\ u \in Y$ with $\|u\| =
\rho$;
\item[$(ii)$]  $\mathcal{J}(0) = 0$ and $\mathcal{J}(e) < \zeta$ for some $e \in Y$
with $\|e\| > \rho$.
\end{itemize}
Let us define $\Gamma = \{\gamma \in C([0, 1]; Y): \gamma(0) = 0,
\gamma(1) =e\}$ and
\begin{displaymath}
c = \inf\limits_{\gamma \in \Gamma }\max\limits_{t \in [0,1]}
\mathcal{J}(\gamma(t)).
\end{displaymath}
Then there exists a sequence $\{u_n\}_n \subset Y$ such that $\mathcal{J}(u_n)
\rightarrow c$ and $\mathcal{J}'(u_n) \rightarrow 0$ in $Y'$ (dual of $Y$).
\end{theorem}
By the assumptions $(V)$, $(M)$ and $(F)$, one can see that
$J_\varepsilon(u)$ has the mountain pass geometry.
\begin{lemma}\label{lemma4.1} Assume that conditions {\rm ($V$)}, $(M)$ and  {\rm ($F$)}  hold.
Then the functional $J_\varepsilon$ satisfies the conclusions
(i)-(ii) of Theorem~\ref{the5.1}.\end{lemma}
\begin{proof}
For each $\varepsilon > 0$, by the fractional  Sobolev embedding, $(M_2)$ and $(f_2)$, we
have
\begin{eqnarray*}
J_\varepsilon(u) &:=&\frac12 \widetilde{M}\left([u]_{s,A}^2\right)+
\frac{\varepsilon^{-2s}}{2}\int_{\mathbb{R}^N}V(x)|u|^2dx-
\frac{\varepsilon^{-2s}}{4}\iint_{\mathbb{R}^{2N}}\frac{F(|u(x)|^2)F(|u(y)|^2)}{|x-y|^{\alpha}}dxdy
- \frac{\varepsilon^{-2s}}{2_s^\ast}\int_{\mathbb{R}^N}
|u|^{2_s^\ast}dx\\
 &\geq& \min\left\{\frac{m_1}{2\sigma},\frac{1}{2}\right
\}\|u\|_\varepsilon^{2\sigma} -
\varepsilon^{-2s}C\|u\|_\varepsilon^{2p}
-\frac{\varepsilon^{-2s}}{2_s^\ast}S^{\frac{-2_s^\ast}{2}}\|u\|_\varepsilon^{2_s^\ast},
\end{eqnarray*}
for all $u \in E$. It follows from $\max\{2, \sigma\} < p$ that
there exist small enough $\varrho_\varepsilon > 0$ and
$\alpha_\varepsilon
> 0$ such that $J_\varepsilon(u) \geq \alpha_\varepsilon > 0$ for
all $u \in E$ with $\|u\|_\varepsilon = \varrho_\varepsilon$, and
all $\varepsilon
> 0$. Hence (i) in Theorem~\ref{the5.1} holds.

Now we verify condition $(ii)$ in Theorem~\ref{the5.1}. Let
$\varphi_0 \in C_0^\infty(\mathbb{R}^N,\mathbb{C})$ with
$\|\varphi_0\|_\varepsilon = 1$. By $(M_2)$, we have
\begin{equation}\label{e5.1}
\widetilde{M}(t)\leq \widetilde{M}(1)t^\sigma \quad \text{for all
}t\geq 1.
\end{equation}
Then by $(f_3)$, the following holds
\begin{eqnarray*}
J_\varepsilon(t\varphi_0) &\leq& \widetilde{M}(1)t^{2\sigma} +
\frac{1}{2}t^{2} -
\frac{\varepsilon^{-2s}}{4}\int_{\mathbb{R}^N}(\mathcal{K}_\alpha*F(|t\varphi_0|^2))F(|t\varphi_0|^2)dx
- \frac{\varepsilon^{-2s}}{2_s^\ast}t^{2_s^\ast}|\varphi_0|_{2_s^\ast}^{2_s^\ast}\\
&\leq& \widetilde{M}(1)t^{2\sigma} + \frac{1}{2}t^{2} -
\frac{\varepsilon^{-2s}}{2_s^\ast}t^{2_s^\ast}|\varphi_0|_{2_s^\ast}^{2_s^\ast},
\end{eqnarray*}
and hence $J_\varepsilon(t\varphi_0) \rightarrow -\infty$ as $t
\rightarrow  \infty$, since $2\sigma < 2_s^\ast$. Therefore, there
exists large enough $t_0$ such that $J_\varepsilon(t_0\varphi_0) <
0$. Then we take $e = t_0\varphi_0$ and $J_\varepsilon(e) < 0$.
Hence $(ii)$ in
Theorem~\ref{the5.1} holds. The proof is thus
complete.
\end{proof}

\section{Verification of $(PS)_c$ condition}

In this section we recall the fractional version of concentration compactness
principle in the fractional Sobolev space, see \cite{PP, XZZ, zhang1} for more details.

\begin{lemma}{ \em (see \cite[Theorem 1.5]{PP})}\label{lemma3.1} Let $\Omega \subseteq \mathbb{R}^N$ be an open subset and let $\{u_n\}_n$ be a sequence  in
$H^s(\mathbb{R}^N)$, weakly converging to $u$ as $n \rightarrow
\infty$ and such that $|u_n|^{2_s^\ast}\rightharpoonup \nu$ and
$|(-\Delta)^{\frac{s}{2}} u_n|^2 \rightharpoonup \mu$ in the sense
of measures. Then either $u_n \rightarrow u$ in
$L_{{\rm loc}}^{2_s^\ast}(\mathbb{R}^N)$ or there exist a (at most
countable) set of distinct points $\{x_j\}_{j \in I} \subseteq
\overline{\Omega}$ and positive numbers $\{\nu_j\}_{j \in I}$ such
that
\begin{eqnarray*}
\nu = |u|^{2_s^\ast} + \sum_{j \in I} \delta_{x_j}\nu_j,\quad \nu_j
> 0.
\end{eqnarray*}
If, in addition, $\Omega$ is bounded, then there exist a positive
measure $\widetilde{\mu} \in \mathcal {M}(\mathbb{R}^N)$ with
$supp\widetilde{\mu} \subseteq \overline{\Omega}$ and positive
numbers $\{\mu_j\}_{j \in I}$ such that
\begin{eqnarray*}
\mu =  |(-\Delta)^{\frac{s}{2}} u|^2 + \widetilde{\mu} + \sum_{j \in
I} \delta_{x_j}\mu_j, \quad  \mu_j > 0
\end{eqnarray*}
and
\begin{eqnarray*}
\nu_j \leq (S^{-1}\mu(\{x_j\}))^{\frac{2_s^\ast}{2}},
\end{eqnarray*}
where $S$ is the best Sobolev constant, i.e.
\begin{eqnarray*} S = \inf\limits_{u \in
H^s(\mathbb{R}^N)}
\frac{\displaystyle\int_{\mathbb{R}^N}|(-\Delta)^{\frac{s}{2}}
u|^2dx}{\displaystyle\int_{\mathbb{R}^N}|u|^{2_s^\ast}dx},
\end{eqnarray*} $x_j \in
\mathbb{R}^N$, $\delta_{x_j}$ are Dirac measures at $x_j$ and
$\mu_j$, $\nu_j$ are constants.
\end{lemma}

In the case $\Omega = \mathbb{R}^N$, the above principle of
 concentration compactness does not provide any information about
 the possible loss of mass at infinity. The following result
 expresses this fact in quantitative terms.

\begin{lemma}{\em (see \cite[Lemma 3.5]{zhang1})}\label{lemma3.2} Let $\{u_n\}_n \subset H^s(\mathbb{R}^N)$ be such that $u_n \rightharpoonup
u$ weakly converges in $H^s(\mathbb{R}^N)$,  $|u_n|^{2_s^\ast}\rightharpoonup
\nu$ and $|(-\Delta)^{\frac{s}{2}} u_n|^2 \rightharpoonup \mu$
weakly-$\ast$ converges in $\mathcal {M}(\mathbb{R}^N)$ and define
\begin{itemize}
\item[$\mathrm{(i)}$]  $\mu_\infty = \lim\limits_{R\rightarrow
\infty}\limsup\limits_{n\rightarrow\infty}\displaystyle\int_{\{x \in
\mathbb{R}^N: |x|>R\}}|(-\Delta)^{\frac{s}{2}} u_n|^{2}dx$,

\item[$\mathrm{(ii)}$]   $\nu_\infty = \lim\limits_{R\rightarrow
\infty}\limsup\limits_{n\rightarrow\infty}\displaystyle\int_{\{x \in
\mathbb{R}^N: |x|>R\}}|u_n|^{2_s^\ast}dx$.
\end{itemize}
Then the quantities $\nu_\infty$ and $\mu_\infty$ exist and satisfy the following
\begin{itemize}
\item[$\mathrm{(iii)}$]  $\limsup\limits_{n\rightarrow\infty}\displaystyle\int_{\mathbb{R}^N}|(-\Delta)^{\frac{s}{2}}
u_n|^{2}dx = \displaystyle\int_{\mathbb{R}^N}d\mu + \mu_\infty$,

\item[$\mathrm{(iv)}$]   $\limsup\limits\limits_{n\rightarrow\infty}\displaystyle\int_{\mathbb{R}^N}|u_n|^{2_s^\ast}dx
= \displaystyle\int_{\mathbb{R}^N}d\nu + \nu_\infty$,

\item[$\mathrm{(v)}$] $\nu_\infty \leq (S^{-1}\nu_\infty)^{\frac{2_s^\ast}{2}}$.
\end{itemize}
\end{lemma}

\indent The main result of this section is the following compactness
result:
\begin{lemma}\label{lemma3.4}Suppose that conditions {\rm ($V$),  ($M$)} and {\rm ($F$)}  hold.  Let  $\{u_n\}_n \subset E$ be a
$(PS)_c$ sequence of functional $J_\varepsilon$, i.e.
\begin{eqnarray*} J_\varepsilon(u_n) \rightarrow c \quad
\mbox{and}\quad J_\varepsilon'(u_n) \rightarrow 0 \quad \mbox{in} \
E'
\end{eqnarray*}
as $n \rightarrow \infty$, where $E'$ is the dual of $E$. Then for
any $0 < \varepsilon < 1$, $J_\varepsilon$ satisfies $(PS)_c$
condition, for all $c \in \left(0,\,
 \sigma_0\varepsilon^{\frac{8s\sigma}{2_s^{\ast}-4\sigma}}\right)$, where $\sigma_0 :=\left(\frac{1}{\mu}-\frac{1}{2_s^\ast}\right)\left(m_1S^{2\sigma}\right)^{\frac{2_s^{\ast}}{2_s^{\ast}-4\sigma}}$,
i.e.
 any $(PS)_c$-sequence $\{u_n\}_n \subset E$ has a strongly
convergent subsequence in $E$.
\end{lemma}
\begin{proof}
If $\inf_{n\geq1}\|u\|_\varepsilon = 0$, then there exists a
subsequence of $\{u_n\}_n$  (still denoted by $\{u_n\}_n$) such that
$u_n \rightarrow 0$ in $E$ as $n \rightarrow \infty$. Thus, we
assume that $d := \inf_{n\geq1}\|u\|_\varepsilon > 0$ in the
 sequel.\\
\indent By $J_\varepsilon(u_n) \rightarrow c$ and
$J_\varepsilon'(u_n) \rightarrow 0$ in $E'$, there exists $C > 0$
such that
\begin{eqnarray}\label{e4.1}
c+ o(1)\|u_n\|_\varepsilon&=&\nonumber J_\varepsilon(u_n) -
\frac{1}{\mu}\langle J_\varepsilon'(u_n), u_n\rangle = \frac12
\widetilde{M}\left([u_n]_{s,A}^2\right)-
\frac{1}{\mu}M\left([u_n]_{s,A}^2\right)[u_n]_{s,A}^2 \\
&&\nonumber
\mbox{}+\left(\frac{1}{2}-\frac{1}{\mu}\right)\varepsilon^{-2s}
\int_{\mathbb{R}^N}V(x)|u_n|^2dx +
\left(\frac{1}{\mu}-\frac{1}{2_s^\ast}\right)\varepsilon^{-2s}\int_{\mathbb{R}^N} |u_n|^{2_s^\ast}dx  \\
&& \mbox{} +
\varepsilon^{-2s}\int_{\mathbb{R}^N}(\mathcal{K}_\alpha*F(|u_n|^2))\left(\frac{1}{\mu}f(|u_n|^2)|u_n|^2-\frac{1}{4}F(|u_n|^2)\right)dx.
\end{eqnarray}
It follows by $(M_2)$ and $(f_3)$ that
\begin{eqnarray*}
C+ C\|u_n\|_\varepsilon&\geq&
\left(\frac{1}{2\sigma}-\frac{1}{\mu}\right)M\left([u_n]_{s,A}^2\right)[u_n]_{s,A}^2
+ \left(\frac{1}{2}-\frac{1}{\mu}\right)\varepsilon^{-2s}
\int_{\mathbb{R}^N}V(x)|u_n|^2dx \\
&\geq&\left(\frac{1}{2\sigma}-\frac{1}{\mu}\right)m_1[u_n]_{s,A}^{2\sigma}
+ \left(\frac{1}{2}-\frac{1}{\mu}\right)\varepsilon^{-2s}
\int_{\mathbb{R}^N}V(x)|u_n|^2dx.
\end{eqnarray*}
This, together with $2 < 2\sigma < 2_s^\ast$,  implies that
$\{u_n\}_n$ is bounded in $E$. Furthermore, we can obtain $c\geq 0$
by passing to the limit in \eqref{e4.1}. Hence, by diamagnetic
inequality, $\{|u_n|\}_n$ is bounded in $H^s(\mathbb{R}^N)$. Therefore
for some subsequence, there is $u \in E$
such that $u_n \rightharpoonup u$ in $E$.\\
\indent  Since $2 < p < \frac{2N-\alpha}{N-2s} < 2_s^\ast$ and $2 <
\frac{4N}{2N-\alpha} < 2_s^\ast$, by Proposition~\ref{pro1} we get that $|u_n|
\rightarrow |u|$ strongly in
$L^{\frac{2Np}{2N-\alpha}}(\mathbb{R}^N)\cap
L^{\frac{4N}{2N-\alpha}}(\mathbb{R}^N)$. Hence the Br\'{e}zis-Lieb
Lemma implies that $u_n \rightarrow u$ strongly in
$L^{\frac{2Np}{2N-\alpha}}(\mathbb{R}^N, \mathbb{C})\cap
L^{\frac{4N}{2N-\alpha}}(\mathbb{R}^N, \mathbb{C})$. By $(f_2)$, we
have
\begin{eqnarray*}
\int_{\mathbb{R}^N}
\left|F(|u_n|^2-F(|u|^2))\right|^{\frac{2N}{2N-\alpha}} dx &\leq&
\int_{\mathbb{R}^N}
\left|f(|u|^2+\varrho(|u_n|^2-|u|^2))\right|^{\frac{2N}{2N-\alpha}}\left||u_n|^2-|u|^2\right|^{\frac{2N}{2N-\alpha}}
dx\\
&\leq& \int_{\mathbb{R}^N}
\left[C(1+(|u_n|+|u|)^{p-2})\right]^{\frac{2N}{2N-\alpha}}\left(|u_n|+|u|\right)^{\frac{2N}{2N-\alpha}}|u_n-u|^{\frac{2N}{2N-\alpha}}
dx\\
&\leq&
C^{{\frac{2N}{2N-\alpha}}}2^{\frac{\alpha}{2N-\alpha}}\int_{\mathbb{R}^N}\left(|u_n|+|u|\right)^{\frac{2N}{2N-\alpha}}|u_n-u|^{\frac{2N}{2N-\alpha}}
dx\\
&&\mbox{}\ +
C^{{\frac{2N}{2N-\alpha}}}2^{\frac{\alpha}{2N-\alpha}}\int_{\mathbb{R}^N}\left(|u_n|+|u|\right)^{(p-1)\frac{2N}{2N-\alpha}}|u_n-u|^{\frac{2N}{2N-\alpha}}
dx.
\end{eqnarray*}
Using the H\"{o}lder inequality, we
can
deduce
\begin{eqnarray*}
\int_{\mathbb{R}^N}
\left|F(|u_n|^2-F(|u|^2))\right|^{\frac{2N}{2N-\alpha}} dx &\leq&
C^{{\frac{2N}{2N-\alpha}}}2^{\frac{\alpha}{2N-\alpha}}\|(|u_n|+|u|)^{\frac{2N}{2N-\alpha}}\|_{L^2(\mathbb{R}^N)}\||u_n-u|^{\frac{2N}{2N-\alpha}}\|_{L^2(\mathbb{R}^N)}\\
&& +
C^{{\frac{2N}{2N-\alpha}}}2^{\frac{\alpha}{2N-\alpha}}\|(|u_n|+|u|)^{(p-1)\frac{2N}{2N-\alpha}}\|_{L^{\frac{p}{p-1}}(\mathbb{R}^N)}\||u_n-u|^{\frac{2N}{2N-\alpha}}\|_{L^p(\mathbb{R}^N)}\\
&\leq& C\||u_n-u|^{\frac{2N}{2N-\alpha}}\|_{L^2(\mathbb{R}^N)} +
C\||u_n-u|^{\frac{2N}{2N-\alpha}}\|_{L^p(\mathbb{R}^N)}\\
&\rightarrow& 0,
\end{eqnarray*}
as $n \rightarrow \infty$, where $C > 0$ is independent of $n$. Thus,
we obtain that $F(|u_n|^2 \rightarrow F(|u|^2)$ in
$L^{\frac{2N}{2N-\alpha}}(\mathbb{R}^N)$. Note that by the
Hardy-Littlewood-Sobolev inequality, the Riesz potential defines a
linear continuous map form $L^{\frac{2N}{2N-\alpha}}(\mathbb{R}^N)$
to $L^{\frac{2N}{\alpha}}(\mathbb{R}^N)$. Then
\begin{eqnarray}\label{e4.3}
\mathcal{K}_\alpha*F(|u_n|^2) \rightarrow
\mathcal{K}_\alpha*F(|u|^2) \quad \mbox{in}\quad
L^{\frac{2N}{\alpha}}(\mathbb{R}^N)
\end{eqnarray}
as $n \rightarrow \infty$.\\
\indent For $\varphi \in E$ fixed, by $(f_2)$ with $\varepsilon = 1$
we have
\begin{eqnarray*}
\int_{\mathbb{R}^N}
\left|f(|u_n|^2)u_n\overline{\varphi}\right|^{\frac{2N}{2N-\alpha}}
dx &\leq&
C^{{\frac{2N}{2N-\alpha}}}2^{\frac{\alpha}{2N-\alpha}}\left(\int_{\mathbb{R}^N}(|u_n||\varphi|)^{\frac{2N}{2N-\alpha}}
dx +
\int_{\mathbb{R}^N}|u_n|^{(p-1)\frac{2N}{2N-\alpha}}|\varphi|^{\frac{2N}{2N-\alpha}}
dx\right)\\
&\leq& C^{{\frac{2N}{2N-\alpha}}}2^{\frac{\alpha}{2N-\alpha}}
\left(\||u_n|^{\frac{2N}{2N-\alpha}}\|_{L^2(\mathbb{R}^N)}\||\varphi|^{\frac{2N}{2N-\alpha}}\|_{L^2(\mathbb{R}^N)}\right.\\
&&\mbox{}\left. +
\||u_n|^{(p-1)\frac{2N}{2N-\alpha}}\|_{L^{\frac{p}{p-1}}(\mathbb{R}^N)}\||\varphi|^{\frac{2N}{2N-\alpha}}\|_{L^p(\mathbb{R}^N)}\right)\\
&\leq& C,
\end{eqnarray*}
thanks to $2 < \frac{4N}{2N-\alpha} < 2_s^\ast$ and $2 <
\frac{2pN}{2N-\alpha} < 2_s^\ast$, where $C > 0$ denotes various
constants.

Clearly, $f(|u_n|^2)u_n\overline{\varphi} \rightarrow
f(|u|^2)u\overline{\varphi}$ a.e. in $\mathbb{R}^N$. Hence, up to a
subsequence, $\mbox{Re}\left\{f(|u_n|^2)u_n\overline{\varphi}\right\}$
weakly converges to $\mbox{Re}\left\{f(|u|^2)u\overline{\varphi}\right\}$ in
$L^{\frac{2N}{2N-\alpha}}(\mathbb{R}^N)$. This together with
\eqref{e4.3} yields that
\begin{eqnarray}\label{e4.4}
\lim_{n \rightarrow \infty}\mbox{Re}\int_{\mathbb{R}^N}
(\mathcal{K}_\alpha*F(|u_n|^2))f(|u_n|^2)u_n\overline{\varphi}dx  =
\lim_{n \rightarrow \infty}\mbox{Re}\int_{\mathbb{R}^N}
(\mathcal{K}_\alpha*F(|u|^2))f(|u|^2)u\overline{\varphi}dx
\end{eqnarray}
for each $\varphi \in E$.\\
\indent We claim that as $n\rightarrow \infty$
\begin{equation}\label{e4.4}
\int_{\mathbb{R}^N} |u_n|^{2_s^\ast} dx \rightarrow
\int_{\mathbb{R}^N} |u|^{2_s^\ast} dx.
\end{equation}
In order to prove this claim, we invoke Prokhorov's Theorem (see Theorem
8.6.2 in \cite{Bog}) to conclude that there exist $\mu, \nu \in \mathcal {M}(\mathbb{R}^N)$
such that
\begin{eqnarray*}
\begin{gathered}
|(-\Delta)^{\frac{s}{2}}
u_n|^2 \rightharpoonup \mu \quad (\text{weak*-sense of measures}),\\
|u_n|^{2_s^\ast}\rightharpoonup \nu\quad (\text{weak*-sense of
measures}),
\end{gathered}
\end{eqnarray*}
where $\mu$ and $\nu$ are a nonnegative  bounded measures on
$\mathbb{R}^N$.  It follows by Lemma \ref{lemma3.1} that
either $u_n \rightarrow u$ in
$L_{loc}^{2_s^\ast}(\mathbb{R}^N)$ or $\nu = |u|^{2_s^\ast} +
\sum_{j \in I} \delta_{x_j}\nu_j$, as $n \rightarrow \infty$, where
$I$ is a countable set, $\{\nu_j\}_j
\subset [0, \infty)$, $\{x_j\}_j \subset \mathbb{R}^N$.\\
\indent Take $\phi \in C_0^\infty(\mathbb{R}^N)$ such that $0 \leq
\phi \leq 1$; $\phi \equiv 1$ in $B(x_j, \rho)$, $\phi(x) = 0$ in
$\mathbb{R}^N \setminus B(x_j, 2\rho)$. For any $\rho
> 0$, define $\phi_\rho =
\phi\left(\frac{x-x_j}{\rho}\right)$, where $j \in I$. It follows
that
\begin{eqnarray}\label{e4.5}
&&\nonumber\iint_{\mathbb{R}^{2N}}\frac{|u_n(x)\phi_\rho(x)-u_n(y)\phi_\rho(y)|^2}{|x-y|^{N+2s}}dxdy\\
&&\nonumber\mbox{}\ \leq
2\iint_{\mathbb{R}^{2N}}\frac{|u_n(x)-u_n(y)|^2\phi_\rho^2(y)}{|x-y|^{N+2s}}dxdy
+
2\iint_{\mathbb{R}^{2N}}\frac{|\phi_\rho(x)-\phi_\rho(y)|^2|u_n(x)|^2}{|x-y|^{N+2s}}dxdy\\
&&\mbox{}\ \leq
2\iint_{\mathbb{R}^{2N}}\frac{|u_n(x)-u_n(y)|^2}{|x-y|^{N+2s}}dxdy
+
2\iint_{\mathbb{R}^{2N}}\frac{|\phi_\rho(x)-\phi_\rho(y)|^2|u_n(x)|^2}{|x-y|^{N+2s}}dxdy.
\end{eqnarray}
Similar to the proof of Lemma 3.4 in \cite{zhang2}, we can show that
\begin{eqnarray}\label{e4.6}
\iint_{\mathbb{R}^{2N}}\frac{|\phi_\rho(x)-\phi_\rho(y)|^2|u_n(x)|^2}{|x-y|^{N+2s}}dxdy
\leq C\rho^{-2s} \int_{B(x_i,K\rho)}|u_n(x)|^2dx + CK^{-N},
\end{eqnarray}
where $K > 4$. Since $\{u_n\}_n$ is bounded in $E$,  it follows from
\eqref{e4.5} and  \eqref{e4.6} that $\{u_n\phi_\rho\}_n$ is bounded in
$E$. Then $\langle J_\varepsilon'(u_n), u_n\phi_\rho\rangle
\rightarrow 0$, which implies
\begin{eqnarray}\label{e4.7}
&&\nonumber
M\left([u_n]_{s,A}^2\right)\iint_{\mathbb{R}^{2N}}\frac{|u_n(x)-e^{i(x-y)\cdot
A(\frac{x+y}{2})}u_n(y)|^2\phi_\rho(y)}{|x-y|^{N+2s}}dx dy +
\varepsilon^{-2s}
\int_{\mathbb{R}^N}V(x)|u_n|^2\phi_\rho(x) dx\\
&&\mbox{}\nonumber = - \mbox{Re}\left\{
M\left([u_n]_{s,A}^2\right)\iint_{\mathbb{R}^{2N}}\frac{(u_n(x)-e^{i(x-y)\cdot
A(\frac{x+y}{2})}u_n(y))\overline{u_n(x)(\phi_\rho(x)-\phi_\rho(y))}}{|x-y|^{N+2s}}dx dy\right\}\\
&&\mbox{}\ \  + \varepsilon^{-2s}
\int_{\mathbb{R}^N}|u_n|^{2_s^\ast}\phi_\rho dx +
\varepsilon^{-2s}\int_{\mathbb{R}^N}\mathcal{K}_\alpha*F(|u_n|^2))f(|u_n|^2)|u_n|^2\phi_\rho
dx +o_n(1).
\end{eqnarray}
Note that by $(M_2)$ and diamagnetic inequality, the following holds
\begin{eqnarray*}
&&
M\left([u_n]_{s,A}^2\right)\iint_{\mathbb{R}^{2N}}\frac{|u_n(x)-e^{i(x-y)\cdot
A(\frac{x+y}{2})}u_n(y)|^2\phi_\rho(y)}{|x-y|^{N+2s}}dx
dy \\
&& \mbox{} \ \ \geq m_1
\left(\iint_{\mathbb{R}^{2N}}\frac{|u_n(x)-e^{i(x-y)\cdot
A(\frac{x+y}{2})}u_n(y)|^2\phi_\rho(y)}{|x-y|^{N+2s}}dx
dy\right)^{2\sigma}\\
&& \mbox{} \ \ \geq m_1
\left(\iint_{\mathbb{R}^{2N}}\frac{\left||u_n(x)|-|u_n(y)|\right|^2\phi_\rho(y)}{|x-y|^{N+2s}}dxdy\right)^{2\sigma}.
\end{eqnarray*}
It is easy to verify that
\begin{eqnarray*}
\iint_{\mathbb{R}^{2N}}\frac{\left||u_n(x)|-|u_n(y)|\right|^2\phi_\rho(y)}{|x-y|^{N+2s}}dxdy\rightarrow
\int_{\mathbb{R}^{N}}\phi_\rho d\mu,
\end{eqnarray*}
as $n \rightarrow \infty$ and
\begin{eqnarray*}
\int_{\mathbb{R}^{N}}\phi_\rho d\mu \rightarrow \mu(\{x_i\})
\end{eqnarray*}
as $\rho \rightarrow 0$. Note that the  H\"{o}lder inequality
implies
\begin{eqnarray}\label{e4.8}
&&\nonumber\left|\mbox{Re}\left\{
M\left([u_n]_{s,A}^2\right)\int\int_{\mathbb{R}^{2N}}\frac{(u_n(x)-e^{i(x-y)\cdot
A(\frac{x+y}{2})}u_n(y))\overline{u_n(x)(\phi_\rho(x)-\phi_\rho(y))}}{|x-y|^{N+2s}}dx dy\right\}\right|\\
&&\nonumber \mbox{} \ \leq
C\iint_{\mathbb{R}^{2N}}\frac{|u_n(x)-e^{i(x-y)\cdot
A(\frac{x+y}{2})}u_n(y)|\cdot|\phi_\rho(x)-\phi_\rho(y)|\cdot|u_n(x)|}{|x-y|^{N+2s}}dxdy
\\
&& \mbox{} \  \leq C
\left(\iint_{\mathbb{R}^{2N}}\frac{|u_n(x)|^2|\phi_\rho(x)-\phi_\rho(y)|^2}{|x-y|^{N+2s}}dxdy\right)^{1/2}.
\end{eqnarray}
\indent  Similar to the proof of Lemma 3.4 in \cite{zhang2}, we can show that
\begin{eqnarray}\label{e4.9}
\lim_{\rho\rightarrow
 0}\lim_{n\rightarrow\infty}\iint_{\mathbb{R}^{2N}}\frac{|u_n(x)|^2|\phi_\rho(x)-\phi_\rho(y)|^2}{|x-y|^{N+2s}}dxdy
=0.
\end{eqnarray}
It follows from
\begin{eqnarray*}
\lim_{n\rightarrow\infty}\int_{\mathbb{R}^N}\mathcal{K}_\alpha*F(|u_n|^2))f(|u_n|^2)|u_n|^2\phi_\rho
dx =
\int_{\mathbb{R}^N}\mathcal{K}_\alpha*F(|u|^2))f(|u|^2)|u|^2\phi_\rho
dx
\end{eqnarray*}
and
\begin{eqnarray*}
\lim_{\rho\rightarrow 0}
\int_{\mathbb{R}^N}\mathcal{K}_\alpha*F(|u|^2))f(|u|^2)|u|^2\phi_\rho
dx = 0
\end{eqnarray*}
that
\begin{eqnarray*}
\lim_{\rho\rightarrow
 0}\lim_{n\rightarrow\infty}\int_{\mathbb{R}^N}\mathcal{K}_\alpha*F(|u_n|^2))f(|u_n|^2)|u_n|^2\phi_\rho
dx =0.
\end{eqnarray*}
Since $\phi_\rho$ has compact support, letting $n\to\infty$ in
\eqref{e4.7}, we can deduce  from \eqref{e4.8}-\eqref{e4.9} and the
diamagnetic inequality that
$$
m_1\left(\mu(\{x_j\})\right)^{2\sigma}\leq \varepsilon^{-2s}\nu_j.
$$
 Combining this fact with Lemma \ref{lemma3.1}, we obtain
$$\nu_j \geq m_1\varepsilon^{2s}S^{2\sigma}
\nu_j^{\frac{4\sigma}{2_s^{\ast}}}.$$ This result implies that
$${\rm (I)} \quad \nu_j = 0 \ \indent \mbox{or}\ \quad {\rm (II)} \quad  \nu_j \geq \left(m_1S^{2\sigma}\right)^{\frac{2_s^{\ast}}{2_s^{\ast}-4\sigma}}\varepsilon^{\frac{2s2_s^{\ast}}{2_s^{\ast}-4\sigma}}.$$
To obtain the possible concentration of mass at infinity, we
similarly define a cut off function $\phi_R \in
C_0^\infty(\mathbb{R}^N)$ such that $\phi_R(x)=0$ on $|x| < R$ and
$\phi_R(x)=1$ on $|x| > R+1$. We can verify that $\{u_n\phi_R\}_n$
is bounded in $E$, hence $\langle J_\varepsilon'(u_n),
u_n\phi_R\rangle \rightarrow 0$,  as $n \rightarrow \infty$, which
implies
\begin{eqnarray}\label{e4.11}
&&\nonumber
M\left([u_n]_{s,A}^2\right)\iint_{\mathbb{R}^{2N}}\frac{|u_n(x)-e^{i(x-y)\cdot
A(\frac{x+y}{2})}u_n(y)|^2\phi_R(y)}{|x-y|^{N+2s}}dx dy
+\varepsilon^{-2s}
\int_{\mathbb{R}^N}V(x)|u_n|^2\phi_R(x) dx\\
&&\mbox{}\nonumber = - \mbox{Re}\left\{
M\left([u_n]_{s,A}^2\right)\iint_{\mathbb{R}^{2N}}\frac{(u_n(x)-e^{i(x-y)\cdot
A(\frac{x+y}{2})}u_n(y))\overline{u_n(x)(\phi_R(x)-\phi_R(y))}}{|x-y|^{N+2s}}dx dy\right\}\\
&&\mbox{}\ \  +
\varepsilon^{-2s}\int_{\mathbb{R}^N}|u_n|^{2_s^\ast}\phi_R dx +
\varepsilon^{-2s}\int_{\mathbb{R}^N}\mathcal{K}_\alpha*F(|u_n|^2))f(|u_n|^2)|u_n|^2\phi_R(x)
dx +o_n(1).
\end{eqnarray}
It is easy to verify that
\begin{eqnarray*}
\limsup\limits_{R\rightarrow\infty}\limsup\limits_{n\rightarrow\infty}\iint_{\mathbb{R}^{2N}}\frac{||u_n(x)|-|u_n(y)||^2\phi_R(y)}{|x-y|^{N+2s}}dxdy
= \mu_\infty
\end{eqnarray*}
and
\begin{eqnarray*}
&& \left|\mbox{Re}\left\{
M\left([u_n]_{s,A}^2\right)\iint_{\mathbb{R}^{2N}}\frac{(u_n(x)-e^{i(x-y)\cdot
A(\frac{x+y}{2})}u_n(y))\overline{u_n(x)(\phi_R(x)-\phi_R(y))}}{|x-y|^{N+2s}}dx
dy\right\}\right| \\
&& \mbox{}  \leq
C\left(\iint_{\mathbb{R}^{2N}}\frac{|u_n(x)|^2|\phi_R(x)-\phi_R(y)|^2}{|x-y|^{N+2s}}dxdy\right)^{1/2}.
\end{eqnarray*}
Note that
\begin{eqnarray*}
&& \limsup\limits_{R \rightarrow \infty}\limsup\limits_{n
\rightarrow \infty}
\iint_{\mathbb{R}^{2N}}\frac{|u_n(x)|^2|\phi_R(x)-\phi_R(y)|^2}{|x-y|^{N+2s}}dxdy
\\
&& \mbox{} = \limsup\limits_{R \rightarrow \infty}\limsup\limits_{n
\rightarrow \infty}
\iint_{\mathbb{R}^{2N}}\frac{|u_n(x)|^2|(1-\phi_R(x))-(1-\phi_R(y))|^2}{|x-y|^{N+2s}}dxdy.
\end{eqnarray*}
Similar to the proof of Lemma 3.4 in \cite{zhang2}, we can show that
\begin{eqnarray*}
\limsup\limits_{R \rightarrow \infty}\limsup\limits_{n \rightarrow
\infty}
\iint_{\mathbb{R}^{2N}}\frac{|u_n(x)|^2|(1-\phi_R(x))-(1-\phi_R(y))|^2}{|x-y|^{N+2s}}dxdy
= 0.
\end{eqnarray*}
It follows from the fact that  $(M_2)$, Lemma \ref{lemma2.1} and
Lemma \ref{lemma3.2} that
\begin{eqnarray*}
&& \limsup\limits_{R \rightarrow \infty}\limsup\limits_{n
\rightarrow
\infty}M\left([u_n]_{s,A}^2\right)\iint_{\mathbb{R}^{2N}}\frac{|u_n(x)-e^{i(x-y)\cdot
A(\frac{x+y}{2})}u_n(y)|^2\phi_R(y)}{|x-y|^{N+2s}}dx
dy \\
&& \mbox{} \ \ \geq \limsup\limits_{R \rightarrow
\infty}\limsup\limits_{n \rightarrow \infty}m_1
\left(\iint_{\mathbb{R}^{2N}}\frac{|u_n(x)-e^{i(x-y)\cdot
A_\varepsilon(\frac{x+y}{2})}u_n(y)|^2\phi_R(y)}{|x-y|^{N+2s}}dx
dy\right)^{2\sigma}\\
&& \mbox{} \ \ \geq \limsup\limits_{R \rightarrow
\infty}\limsup\limits_{n \rightarrow \infty} m_1
\left(\iint_{\mathbb{R}^{2N}}\frac{\left||u_n(x)|-|u_n(y)|\right|^2\phi_R(y)}{|x-y|^{N+2s}}dxdy\right)^{2\sigma}
= m_1 \mu_\infty^{2\sigma}.
\end{eqnarray*}
It is easy to see that
\begin{eqnarray*}
\lim_{R \rightarrow \infty}\lim_{n \rightarrow
\infty}\int_{\mathbb{R}^N}\mathcal{K}_\alpha*F(|u_n|^2))f(|u_n|^2)|u_n|^2\phi_R(x)
dx = 0.
\end{eqnarray*}
By Lemma \ref{lemma3.2} and letting $R \to \infty$ in \eqref{e4.11},
we obtain
$$\nu_\infty \geq m_1\varepsilon^{2s}S^{2\sigma}
\nu_\infty^{\frac{4\sigma}{2_s^{\ast}}}.$$ This result implies that
$${\rm (III)} \quad \nu_\infty = 0 \ \indent \mbox{or}\ \quad {\rm (IV)} \quad  \nu_\infty \geq \left(m_1S^{2\sigma}\right)^{\frac{2_s^{\ast}}{2_s^{\ast}-4\sigma}}\varepsilon^{\frac{2s2_s^{\ast}}{2_s^{\ast}-4\sigma}}.$$
\indent Next, we claim that $(II)$ and $(IV)$ cannot occur. If the
case $(IV)$ holds for some $j \in I$, then by Lemma \ref{lemma3.2},
$(M)$ and $(H)$, we have
\begin{eqnarray*}
c &=& \lim_{n \rightarrow \infty}\left(J_\varepsilon(u_n) -
\frac{1}{\mu}\langle
J'_\varepsilon(u_n), u_n\rangle\right)\\
&\geq&  \left(\frac{1}{2\sigma}-\frac{1}{\mu}\right)
M\left([u_n]_{s,A}^2\right)[u_n]_{s,A}^2
+\left(\frac{1}{2}-\frac{1}{\mu}\right)\varepsilon^{-2s}
\int_{\mathbb{R}^N}V(x)|u_n|^2dx   \\
 &&\nonumber \mbox{} +
\left(\frac{1}{\mu}-\frac{1}{2_s^\ast}\right)\varepsilon^{-2s}\int_{\mathbb{R}^N} |u_n|^{2_s^\ast}dx
+ \varepsilon^{-2s} \int_{\mathbb{R}^N}(\mathcal{K}_\alpha*F(|u_n|^2))\left(\frac{1}{\mu}f(|u_n|^2)|u_n|^2-\frac{1}{4}F(|u_n|^2)\right)dx\\
&\geq&
\left(\frac{1}{\mu}-\frac{1}{2_s^\ast}\right)\varepsilon^{-2s}\int_{\mathbb{R}^N}
|u_n|^{2_s^\ast}dx
 \geq
\left(\frac{1}{\mu}-\frac{1}{2_s^\ast}\right)\varepsilon^{-2s}\nu_\infty
\\
&\geq&
\left(\frac{1}{\mu}-\frac{1}{2_s^\ast}\right)\left(m_1S^\sigma\right)^{\frac{2_s^{\ast}}{2_s^{\ast}-4\sigma}}\varepsilon^{\frac{4s\sigma}{2_s^{\ast}-4\sigma}}
= \sigma_0\varepsilon^{\frac{8s\sigma}{2_s^{\ast}-4\sigma}},
\end{eqnarray*}
where $\sigma_0 =
\left(\frac{1}{\mu}-\frac{1}{2_s^\ast}\right)\left(m_1S^\sigma\right)^{\frac{2_s^{\ast}}{2_s^{\ast}-2\sigma}}$,
which is impossible.

Consequently, $\nu_j = 0$ for all $j\in I$.
Similarly, we can prove that $(II)$ cannot occur for any $j$. Thus
\begin{eqnarray} \label{e3.11}
\int_{\mathbb{R}^N}|u_n|^{2_s^\ast}dx \rightarrow
\int_{\mathbb{R}^N}|u|^{2_s^\ast}dx.
\end{eqnarray}
The Br\'{e}zis-Lieb Lemma implies that
\begin{eqnarray*}
\lim\limits_{n \rightarrow
\infty}\int_{\mathbb{R}^N}|u_n-u|^{2_s^\ast}dx = 0.
\end{eqnarray*}
Therefore, we get
\begin{eqnarray*}
u_n \rightarrow u \quad \mbox{in}\quad L^{2_s^\ast}(\mathbb{R}^N)
\quad \mbox{as}\quad n \rightarrow \infty.
\end{eqnarray*}
By the weak lower semicontinuity of the norm,
conditon $(m_1)$ and the Br\'{e}zis-Lieb Lemma, we have
\begin{eqnarray*}
 o(1)\|u_n\|_\varepsilon &=& \nonumber\langle J_\varepsilon'(u_n), u_n\rangle =  M\left([u_n]_{s,A}^2\right)[u_n]_{s,A}^2
+ \varepsilon^{-2s}\int_{\mathbb{R}^N} V(x)|u_n|^2dx \\
&&\mbox{} - \varepsilon^{-2s}\int_{\mathbb{R}^N}|u_n|^{2_s^\ast}dx
- \varepsilon^{-2s}\int_{\mathbb{R}^N}\mathcal{K}_\alpha*F(|u_n|^2)f(|u_n|^2)|u_n|^2dx\\
&\geq& m_1\left([u_n]_{s,A}^{2\sigma} - [u]_{s,A}^{2\sigma}\right)
+ \varepsilon^{-2s}\int_{\mathbb{R}^N} V(x)(|u_n|^2-|u|^2)dx + M\left([u]_{s,A}^2\right)[u]_{s,A}^2\\
&&\mbox{}+ \varepsilon^{-2s}\int_{\mathbb{R}^N} V(x)|u|^2dx -
 \varepsilon^{-2s}\int_{\mathbb{R}^N}|u|^{2_s^\ast}dx -
\varepsilon^{-2s}\int_{\mathbb{R}^N}\mathcal{K}_\alpha*F(|u|^2)f(|u|^2)|u|^2dx\\
&\geq& \min\{m_1,1\}\min\{\|u_n - u\|_\varepsilon^{2\sigma}, \|u_n -
u\|_\varepsilon^2\}    + o(1)\|u\|_\varepsilon.
\end{eqnarray*}
Here we use the fact that $J_\varepsilon'(u) = 0$.  Thanks to $2 <
2\sigma$, we have proved that $\{u_n\}_n$ strongly converges to $u$
in $E$. Hence the proof is complete.
\end{proof}

\section{Proofs of Main  Theorems}
\indent In this section, we shall prove our main results.  We shall first establish Theorem~\ref{the3.1}.

Note that
$J_\varepsilon(u)$ does not satisfy $(PS)_c$ condition for any $c >
0$. Thus, in the sequel we shall find a special finite-dimensional
subspace by which we construct sufficiently
small minimax levels. \\
\indent Recall that the assumption $(V)$ implies that there is $x_0
\in \mathbb{R}^N$ such that $V(x_0) = \min_{x\in \mathbb{R}^N} V(x)
= 0$. Without loss of generality we can assume from now on that $x_0
= 0$.
\begin{proposition} {\rm (see \cite[Theorem 3.2]{zhang3})}\label{pro4.1}
For any $q \in (2, 2_s^\ast)$, we have
\begin{displaymath}
\inf\left\{\iint_{\mathbb{R}^{2N}}\frac{|\phi(x)-\phi(y)|^2}{|x-y|^{N+2s}}dxdy:
\phi \in C_0^\infty (\mathbb{R}^N), |\phi|_q = 1\right\} = 0.
\end{displaymath}
\end{proposition}
By Proposition~\ref{pro4.1},   one can choose $\phi_\zeta \in
C_0^\infty (\mathbb{R}^N)$ with $|\phi_\zeta|_q = 1$ and
supp\,$\phi_\zeta \subset B_{r_\zeta} (0)$ so that
\begin{displaymath}
\iint_{\mathbb{R}^{2N}}\frac{|\phi_\zeta(x)-\phi_\zeta(y)|^2}{|x-y|^{N+2s}}dxdy
\leq C\zeta^{\frac{2N-(N-2s)q}{q}},
\end{displaymath}
for any $1 > \zeta > 0$.\\
Set
\begin{equation}\label{e5.2}
\psi_\zeta(x) = e^{iA(0)x}\phi_\zeta(x), \quad
\psi_{\varepsilon,\zeta}(x) = \psi_\zeta(\varepsilon^{-\tau}x)
\end{equation}
and
\begin{equation}\label{e5.3}
\tau := \frac{2s2_s^{\ast}}{N(2_s^{\ast}-4\sigma)}.
\end{equation}

By $(f_3)$, for any $t>0$ we get
\begin{eqnarray*}
J_\varepsilon(t\psi_{\varepsilon,\zeta})
&\leq&\frac{C_0}{2}t^{2\sigma}\left(\iint_{\mathbb{R}^{2N}}\frac{|\psi_{\varepsilon,\zeta}(x)-e^{i(x-y)\cdot
A(\frac{x+y}{2})}\psi_{\varepsilon,\zeta}(y)|^2}{|x-y|^{N+2s}}dx dy\right)^{2\sigma} \\
&& \mbox{} + \frac{t^2}{2}\varepsilon^{-2s}
\int_{\mathbb{R}^N}V(x)|\psi_{\varepsilon,\zeta}|^2dx  -
t^{2_s^\ast}\frac{\varepsilon^{-2s}}{2_s^\ast}\int_{\mathbb{R}^N}
|u|^{2_s^\ast}dx\\
&\leq&
\varepsilon^{N\tau-2s}\left[\frac{C_0}{2}t^{2\sigma}\left(\iint_{\mathbb{R}^{2N}}\frac{|\psi_{\zeta}(x)-e^{i(\varepsilon^{\tau}x-\varepsilon^{\tau}y)\cdot
A(\frac{\varepsilon^{\tau} x+\varepsilon^{\tau}y}{2})}\psi_{\zeta}(y)|^2}{|x-y|^{N+2s}}dx dy\right)^{2\sigma} \right.\\
&& \mbox{} \left. +
\frac{t^2}{2}\int_{\mathbb{R}^N}V\left(\varepsilon^{\tau}
x\right)|\psi_{\zeta}|^2dx-
 \frac{t^{2_s^\ast}}{2_s^\ast}\int_{\mathbb{R}^N} |\psi_{\zeta}|^{2_s^\ast}dx\right]\\
&=& \varepsilon^{\frac{8s\sigma}{2_s^{\ast}-4\sigma}}
I_\varepsilon(t\psi_{\zeta}),
\end{eqnarray*}
where $I_\varepsilon \in C^1(E, \mathbb{R})$ is defined by
\begin{eqnarray*}
I_\varepsilon(u) &:=&
\frac{C_0}{2}\left(\iint_{\mathbb{R}^{2N}}\frac{|u(x)-e^{i(\varepsilon^{\tau}x-\varepsilon^{\tau}y)\cdot
A(\frac{\varepsilon^{\tau}x+\varepsilon^{\tau}
y}{2})}u(y)|^2}{|x-y|^{N+2s}}dx
dy\right)^{2\sigma} \\
&& \mbox{} + \frac{1}{2}\int_{\mathbb{R}^N}V\left(\varepsilon^{\tau}
x\right)|u|^2dx - \frac{1}{2_s^\ast}\int_{\mathbb{R}^N}
|u|^{2_s^\ast}dx.
\end{eqnarray*}
Since $2_s^\ast > 2\sigma$,  there exists a finite number $t_0 \in
[0, +\infty)$ such that
\begin{eqnarray*}
\max_{t \geq 0} I_\varepsilon(t\psi_{\zeta}) &=&
\frac{C_0}{2}t_0^{2\sigma}\left(\iint_{\mathbb{R}^{2N}}\frac{|\psi_{\zeta}(x)-e^{i(\varepsilon^{\tau}x-\varepsilon^{\tau}y)\cdot
A(\frac{\varepsilon^{\tau}x+\varepsilon^{\tau}
y}{2})}\psi_{\zeta}(y)|^2}{|x-y|^{N+2s}}dx
dy\right)^{2\sigma}  \\
&& \mbox{} +
\frac{t_0^2}{2}\int_{\mathbb{R}^N}V\left(\varepsilon^{\tau}
x\right)|\psi_{\zeta}|^2dx-
\frac{t_0^{2_s^\ast}}{2_s^\ast}\int_{\mathbb{R}^N} |\psi_{\zeta}|^{2_s^\ast}dx\\
&\leq&
\frac{C_0}{2}t_0^{2\sigma}\left(\iint_{\mathbb{R}^{2N}}\frac{|\psi_{\zeta}(x)-e^{i(\varepsilon^{\tau}x-\varepsilon^{\tau}y)\cdot
A(\frac{\varepsilon^{\tau}x+\varepsilon^{\tau}
y}{2})}\psi_{\zeta}(y)|^2}{|x-y|^{N+2s}}dx
dy\right)^{2\sigma} \\
&& \mbox{} +
\frac{t_0^2}{2}\int_{\mathbb{R}^N}V\left(\varepsilon^{\tau}
x\right)|\psi_\zeta|^2dx.
\end{eqnarray*}
Let $\psi_\zeta(x) = e^{iA(0)x}\phi_\zeta(x)$, where $\phi_\zeta(x)$
is as defined above. Then we have the following lemma.
\begin{lemma}{\rm (see \cite[Lemma 3.6]{zhang3})}\label{lemma5.3}(Norm estimate)   For any $\zeta > 0$ there exists $\varepsilon_0 = \varepsilon_0(\zeta) >
0$ such that
\begin{eqnarray*}
\iint_{\mathbb{R}^{2N}}\frac{|\psi_{\zeta}(x)-e^{i(\varepsilon^{\tau}x-\varepsilon^{\tau}y)\cdot
A(\frac{\varepsilon^{\tau}x+\varepsilon^{\tau}
y}{2})}\psi_{\zeta}(y)|^2}{|x-y|^{N+2s}}dx dy \leq
C\zeta^{\frac{2N-(N-2s)q}{q}} + \frac{1}{1-s}\zeta^{2s} +
\frac{4}{s}\zeta^{2s},
\end{eqnarray*}
for all $0 < \varepsilon < \varepsilon_0$ and some constant $C > 0$
depending only on $[\phi]_{s,0}$.
\end{lemma}
On the one hand, since $V(0) = 0$ and note that supp\,$\phi_\zeta
\subset B_{r_\zeta}(0)$, there is $\varepsilon^\ast > 0$ such that
\begin{displaymath}
V\left(\varepsilon^{\tau} x\right) \leq
\frac{\zeta}{|\phi_\zeta|_2^2}\quad \mbox{for\ all }\ |x| \leq
r_\zeta\ \mbox{and}\ 0 < \varepsilon < \varepsilon^\ast.
\end{displaymath}
This implies that
\begin{equation}\label{e5.5}
\max_{t\geq 0}I_\varepsilon(t\phi_\delta) \leq
\frac{C_0}{2}t_0^{2\sigma}\left(C\zeta^{\frac{2N-(N-2s)q}{q}} +
\frac{1}{1-s}\zeta^{2s} + \frac{4}{s}\zeta^{2s}\right)^{2\sigma} +
\frac{t_0^2}{2}\zeta.
\end{equation}
Therefore, for all $0 < \varepsilon <
\min\{\varepsilon_0,\varepsilon^\ast\}$, we have
\begin{equation}\label{e5.6}
\max_{t\geq 0} J_\varepsilon(t\psi_{\lambda,\zeta}) \leq
\left[\frac{C_0}{2}t_0^{2\sigma}\left(C\zeta^{\frac{2N-(N-2s)q}{q}}
+ \frac{1}{1-s}\zeta^{2s} + \frac{4}{s}\zeta^{2s}\right)^{2\sigma} +
\frac{t_0^2}{2}\zeta\right]\varepsilon^{\frac{8s\sigma}{2_s^{\ast}-4\sigma}}.
\end{equation}
\indent Thus we have the following result.
\begin{lemma}\label{lemma4.3} Under the assumptions of Lemma \ref{lemma4.1},
for any $\kappa > 0$ there exists $\mathcal {E}_\kappa > 0$ such
that for each $0 < \varepsilon < \mathcal {E}_\kappa$, there is
$\widehat{e}_\varepsilon \in E$ with $\|\widehat{e}_\varepsilon\| >
\varrho_\varepsilon$, $J_\varepsilon(\widehat{e}_\varepsilon) \leq
0$ and
\begin{equation}\label{e5.7}
\max_{t\in [0, 1]} J_\varepsilon(t\widehat{e}_\varepsilon) \leq
\kappa\varepsilon^{\frac{8s\sigma}{2_s^{\ast}-4\sigma}}.
\end{equation}
\end{lemma}
\begin{proof}
Choose $\zeta > 0$ so small that
\begin{displaymath}
\frac{C_0}{2}t_0^{2\sigma}\left(C\zeta^{\frac{2N-(N-2s)q}{q}} +
\frac{1}{1-s}\zeta^{2s} + \frac{4}{s}\zeta^{2s}\right)^{2\sigma} +
\frac{t_0^2}{2}\zeta \leq \kappa.
\end{displaymath}
Let $\psi_{\varepsilon,\zeta} \in E$ be the function defined by
\eqref{e5.3}. Set $\mathcal {E}_\kappa =
\min\{\varepsilon_0,\varepsilon^\ast\}$. Let
$\widehat{t}_\varepsilon > 0$ be such that
$\widehat{t}_\varepsilon\|\psi_{\varepsilon,\zeta}\|_\varepsilon
> \varrho_\varepsilon$ and $J_\varepsilon(t\psi_{\varepsilon,\zeta}) \leq 0$ for
all $t \geq \widehat{t}_\varepsilon$. By \eqref{e5.6}, let
$\widehat{e}_\varepsilon = \widehat{t}_\varepsilon
\psi_{\varepsilon,\zeta}$ we know that the conclusion of Lemma
\ref{lemma4.3} holds.
\end{proof}

\noindent{\bf Proof of Theorem~\ref{the3.1}.} For any $0 < \kappa <
\sigma_0$, by Lemma \ref{lemma3.4}, we choose $\mathcal {E}_\kappa >
0$ and define for $0 < \varepsilon < \mathcal {E}_\kappa$, the
minimax value
$$c_\varepsilon := \inf_{\gamma \in \Gamma_\varepsilon}\max_{t\in [0,1]} J_\varepsilon(t\widehat{e}_\varepsilon),$$
where $$\Gamma_\varepsilon := \{\gamma \in C([0, 1], E): \gamma(0) =
0 \ \mbox{and}\ \gamma(1) = \widehat{e}_\varepsilon\}.$$ By Lemma
\ref{lemma4.1}, we have $\alpha_\varepsilon \leq c_\varepsilon \leq
\kappa\varepsilon^{\frac{8s\sigma}{2_s^{\ast}-4\sigma}}$. By virtue
of Lemma \ref{lemma3.4}, we know that $J_\varepsilon$ satisfies the
$(PS)_{c_\lambda}$ condition, there is $u_\varepsilon \in E$ such
that $J'_\varepsilon(u_\varepsilon) = 0$ and
$J_\varepsilon(u_\varepsilon) = c_\varepsilon$, Then $u_\varepsilon$
is a nontrivial mountain pass solution  of problem  \eqref{e3.1}.

 Since $u_\varepsilon$ is a critical point of
$J_\varepsilon$, by $(M)$ and $(H)$, we have for $\tau \in [2\sigma,
2_s^\ast]$
\begin{eqnarray}\label{e5.11}
\kappa\varepsilon^{\frac{8s\sigma}{2_s^{\ast}-4\sigma}}
&\geq&\nonumber J_\varepsilon (u_\varepsilon) =
J_\varepsilon(u_\varepsilon) - \frac{1}{\tau}\langle
J_\varepsilon'(u_\varepsilon), u_\varepsilon\rangle
\\
&=&\nonumber \frac12
\widetilde{M}\left([u_\varepsilon]_{s,A_\varepsilon}^2\right)-
\frac{1}{\tau}
M\left([u_\varepsilon]_{s,A_\varepsilon}^2\right)[u_\varepsilon]_{s,A_\varepsilon}^2+\left(\frac{1}{2}-\frac{1}{\tau}\right)\varepsilon^{-2s}
\int_{\mathbb{R}^N}V(x)|u_\varepsilon|^2dx \\
&&\nonumber \mbox{}+
\left(\frac{1}{\tau}-\frac{1}{2_s^\ast}\right)\varepsilon^{-2s}\int_{\mathbb{R}^N}
|u_\varepsilon|^{2_s^\ast}dx +
\varepsilon^{-2s}\int_{\mathbb{R}^N}(\mathcal{K}_\alpha*F(|u_\varepsilon|^2))\left(\frac{1}{\tau}f(|u_\varepsilon|^2)|u_\varepsilon|^2-\frac{1}{4}F(|u_\varepsilon|^2)\right)dx\\
&\geq&\nonumber
\left(\frac{1}{2\sigma}-\frac{1}{\tau}\right)m_1[u_\varepsilon]_{s,A_\varepsilon}^{2\sigma}
+ \left(\frac{1}{2}-\frac{1}{\tau}\right)\varepsilon^{-2s}
\int_{\mathbb{R}^N}V(x)|u_\varepsilon|^2dx\\
&& \mbox{}+
\left(\frac{1}{\tau}-\frac{1}{2_s^\ast}\right)\varepsilon^{-2s}\int_{\mathbb{R}^N}
|u_\varepsilon|^{2_s^\ast}dx +
\left(\frac{\mu}{\tau}-\frac{1}{4}\right)\varepsilon^{-2s}
\iint_{\mathbb{R}^{2N}}\frac{F(|u_\varepsilon(x)|^2)F(|u_\varepsilon(y)|^2)}{|x-y|^{\alpha}}dxdy.
\end{eqnarray}
Taking $\tau = 2/\sigma$, we obtain the estimate $\eqref{e1.8}$ and
taking $\tau = \mu$ we obtain the estimate $\eqref{e1.9}$. This
completes the proof of Theorem~\ref{the3.1}. $\hfill\Box$

Next, we shall establish Theorem~\ref{the3.2}. Again, we shall first need to prove a lemma.

For any $m^{\ast} \in \mathbb{N}$, one can choose $m^{\ast}$
functions $\phi_\zeta^i \in C_0^\infty(\mathbb{R}^N)$ such that
supp\,$\phi_\zeta^i$ $ \cap$ supp\,$\phi_\zeta^k = \emptyset$, $i
\neq k$, $|\phi_\zeta^i|_s = 1$ and
\begin{displaymath}
\iint_{\mathbb{R}^{2N}}\frac{|\phi_\zeta^i(x)-\phi_\zeta^i(y)|^2}{|x-y|^{N+2s}}dxdy
\leq C\zeta^{\frac{2N-(N-2s)q}{q}}.
\end{displaymath}
Let $r_\zeta^{m^{\ast}}
> 0$ be such that supp\,$\phi_\zeta^{i} \subset B_{r_\zeta}^{i}(0)$
for $i = 1,2,\cdots,m^{\ast}$. Set
\begin{equation}\label{e5.8}
\psi_\zeta^i(x) = e^{iA(0)x}\phi_\zeta^i(x)
\end{equation}
and
\begin{equation}\label{e5.9}
\psi_{\varepsilon,\zeta}^i(x) = \psi_\zeta^i(\varepsilon^{-1}x).
\end{equation}
Denote
\begin{displaymath}
\mathcal{H}_{\varepsilon\zeta}^{m^{\ast}} =
\mbox{span}\{\psi_{\varepsilon,\zeta}^1, \psi_{\varepsilon,\zeta}^2,
\cdots, \psi_{\varepsilon,\zeta}^{m^{\ast}}\}.
\end{displaymath}
Observe that for each $u = \displaystyle\sum_{i=1}^{m^{\ast}}c_i
\psi_{\varepsilon,\zeta}^i \in \mathcal{H}_{\varepsilon\zeta}^{m^{\ast}}$, we
have
\begin{displaymath}
[u]_{s,A_\varepsilon}^2 \leq
C\displaystyle\sum_{i=1}^{m^{\ast}}|c_i|^2[\psi_{\varepsilon,\zeta}^i]_{s,A_\varepsilon}^2,
\end{displaymath}
for some constant $C > 0$. Therefore
\begin{displaymath}
J_\varepsilon(u) \leq  C\sum_{i=1}^{m^{\ast}}J_{\varepsilon}(c_i
\psi_{\varepsilon,\zeta}^i)
\end{displaymath}
for some constant $C > 0$. Based on a similar argument as before, we
see that
\begin{displaymath}
J_\varepsilon(c_i \psi_{\varepsilon,\zeta}^i) \leq \varepsilon^{N -
2s}\Psi(|c_i|\psi_{\zeta}^i).
\end{displaymath}
As before, we can obtain the following estimate:
\begin{equation}\label{e5.10}
\max_{u\in \mathcal{H}_{\varepsilon\delta}^{m^{\ast}}} J_\varepsilon(u) \leq C
m^\ast
\left[\frac{C_0}{2}t_0^{2\sigma}\left(C\zeta^{\frac{2N-(N-2s)q}{q}}
+ \frac{1}{1-s}\zeta^{2s} + \frac{4}{s}\zeta^{2s}\right)^{2\sigma} +
\frac{t_0^2}{2}\zeta\right]\varepsilon^{\frac{8s\sigma}{2_s^{\ast}-4\sigma}}
\end{equation}
for all small enough $\zeta$ and some constant $C > 0$. From the estimate \eqref{e5.10} we have the following:
\begin{lemma}\label{lemma4.4} Under the assumptions of Lemma \ref{lemma4.1},
for any $m^{\ast} \in \mathbb{N}$ and $\kappa > 0$ there exists
$\mathcal {E}_{m^{\ast}\kappa} > 0$ such that for each $0 <
\varepsilon < \mathcal {E}_{m^{\ast}\kappa}$, there exists an
$m^{\ast}$-dimensional subspace $\mathcal{F}_{\lambda m^{\ast}}$ satisfying
\begin{displaymath}
\max_{u\in \mathcal{F}_{\lambda m^{\ast}}} J_\varepsilon(u) \leq
\kappa\varepsilon^{\frac{8s\sigma}{2_s^{\ast}-4\sigma}}.
\end{displaymath}
\end{lemma}
\begin{proof}
Choose $\zeta > 0$ so small that
\begin{displaymath}
 C m^\ast
\left[\frac{C_0}{2}t_0^{2\sigma}\left(C\zeta^{\frac{2N-(N-2s)q}{q}}
+ \frac{1}{1-s}\zeta^{2s} + \frac{4}{s}\zeta^{2s}\right)^{2\sigma} +
\frac{t_0^2}{2}\zeta\right] \leq \kappa.
\end{displaymath}
Set $\mathcal{F}_{\varepsilon m^{\ast}} =
\mathcal{H}_{\varepsilon\zeta}^{m^{\ast}}=\mbox{span}\{\psi_{\varepsilon,\zeta}^1,
\psi_{\varepsilon,\zeta}^2, \cdots,
\psi_{\varepsilon,\zeta}^{m^{\ast}}\}$. Now the conclusion of Lemma \ref{lemma4.4} follows from \eqref{e5.10}.
\end{proof}

\vspace{2mm}

\noindent{\bf Proof of Theorem~\ref{the3.2}.}
 Denote the set of all symmetric
(in the sense that $-Z = Z$) and closed subsets of $E$ by $\Sigma$,
for each $Z \in \Sigma$. Let gen$(Z)$ be the Krasnoselski genus and
\begin{displaymath}
j(Z) := \min_{\eta\in
\Phi_{m^\ast}}\mbox{gen}(\eta(Z)\cap\partial
B_{\varrho_\varepsilon}),
\end{displaymath}
where $\Phi_{m^\ast}$ is the set of all odd homeomorphisms $\eta
\in C(E, E)$ and $\varrho_\varepsilon$ is the number from Lemma
\ref{lemma4.1}. Then $j$ is a version of Benci's pseudoindex
\cite{b1}. Let
\begin{displaymath}
c_{\varepsilon i} := \inf_{j(Z)\geq i}\sup_{u\in Z}J_\varepsilon(u),
\quad 1 \leq i \leq m^\ast.
\end{displaymath}
Since $J_\varepsilon(u) \geq \alpha_\varepsilon$ for all $u \in
\partial B_{\varrho_\varepsilon}^{+}$ and since $j(\mathcal{F}_{\varepsilon m^\ast}) =
\dim \mathcal{F}_{\varepsilon m^\ast} = m^\ast$, we obtain
\begin{displaymath}
\alpha_\varepsilon \leq c_{\varepsilon 1} \leq \cdots\leq
c_{\varepsilon m^\ast} \leq \sup_{u \in H_{\varepsilon m^\ast}}
J_\varepsilon(u) \leq
\kappa\varepsilon^{\frac{8s\sigma}{2_s^{\ast}-4\sigma}}.
\end{displaymath}
It follows from Lemma \ref{lemma3.4}  that $J_\varepsilon$ satisfies
the $(PS)_{c_\varepsilon}$ condition at all levels $c <
\sigma_0\varepsilon^{N-2s}$. By the usual critical point theory, all
$c_{\varepsilon i}$ are critical levels and $J_\varepsilon$ has at
least $m^\ast$ pairs of nontrivial critical points satisfying
\begin{displaymath}
\alpha_\varepsilon \leq J_\varepsilon(u_\varepsilon) \leq
\kappa\varepsilon^{\frac{8s\sigma}{2_s^{\ast}-4\sigma}}.
\end{displaymath}
Hence, problem \eqref{e3.1} has at least $m^\ast$ pairs of
solutions. Finally, as in the proof of Theorem~\ref{the3.1}, we see that
these solutions satisfy the estimates  $\eqref{e1.8}$ and
$\eqref{e1.9}$. $\hfill\Box$

\section*{Competing Interests}
The authors declare that this work does not represent any conflict of interest.

\section*{Acknowledgements}
We thank the referee for useful remarks.
The first author was supported by the Research Foundation during the 13th Five-Year Plan Period of Department of Education of Jilin Province (JJKH20170648KJ), Natural Science Foundation of
Changchun Normal University (No. 2017-09). 
The second author was supported by the Slovenian Research Agency
(No. P1-0292, N1-0114, N1-0083, N1-0064, and J1-8131).
The third author was  supported by the National Natural Science Foundation of China (No. 11871199)
and Heilongjiang Province
Postdoctoral Startup Foundation (LBH-Q18109).

\end{document}